\documentclass[letterpaper,12pt]{article}

\usepackage{amsmath}
\usepackage{amsfonts,amsthm,amssymb}
\usepackage{eufrak}
\usepackage{fancyhdr}
\usepackage[all]{xy}

\newtheorem{thm}{Theorem}[section]

\newtheorem{prop}[thm]{Proposition}
\newtheorem{cor}[thm]{Corollary}

\newtheorem{lem}[thm]{Lemma}

\renewcommand{\part}{\partial}
\newcommand{\deq}{\mathrel{\mathop:}=}
\newcommand{\Set}{\mathbf{Set}}

\newcommand{\Grp}{\mathbf{Grp}}

\newcommand{\FinGrp}{\mathbf{FinGrp}}

\newcommand{\mc}[1]{\mathcal{#1}}

\newcommand{\sk}{\mathrm{sk}}
\newcommand{\cosk}{\mathrm{cosk}}

\renewcommand{\lim}{\varprojlim}
\newcommand{\colim}{\varinjlim}
\newcommand{\et}[1]{\acute{e}t({#1})}
\newcommand{\finet}[1]{\mathbf{Fin\acute{e}t}({#1})}

\newcommand{\Shv}[1]{\mathbf{Shv}({#1})}

\newcommand{\Mor}[1]{\mathrm{Mor}({#1})}

\newcommand{\Gal}[1]{\mathrm{Gal}({#1})}
\newcommand{\Gals}{\mathrm{Gal}}
\newcommand{\ets}{\acute{e}t}

\newcommand{\Aut}{\mathrm{Aut}}

\newcommand{\Hom}{\mathrm{Hom}}

\newcommand{\Tors}{\mathrm{Tors}}
\newcommand{\Spec}[1]{\textrm{Spec}\,{#1}}

\newcommand{\khorn}{\Lambda^n_k}
\newcommand{\nsimp}{\Delta^n}
\newcommand{\dnsimp}{\part\nsimp}

\newcommand{\Exinfn}{\mathrm{Ex}^{\infty}}

\newcommand{\ccomp}{\scriptstyle{\Pi}\displaystyle}

\newcommand{\rars}{\rightrightarrows}
\newcommand{\pr}{\mathrm{pr}}

\newcommand{\loc}{\mathrm{loc.}}

\newcommand{\hra}{\hookrightarrow}

\newcommand{\shpi}{\tilde{\pi}}

\newcommand{\HR}{\mathrm{HR}}

\newcommand{\epi}{\twoheadrightarrow}

\newcommand{\Gtor}{G\textrm{-}\Tors}
\newcommand{\Htor}{H\textrm{-}\Tors}

\newcommand{\SLC}{\mathbf{SLC}}

\begin{document}
\title{Nonabelian $H^1$ and the \'Etale van Kampen Theorem}
\author{Michael D. Misamore\\Universit\"at Duisburg-Essen\\Universit\"atsstr. 2,\\45117 Essen\\Germany}
\date{October 28, 2009}
\maketitle

\begin{abstract}
Generalized \'etale homotopy pro-groups $\pi_1^{\ets}(\mc{C}, x)$ associated to pointed connected small Grothendieck
sites $(\mc{C}, x)$ are defined and their relationship to Galois theory and the theory of pointed torsors for discrete
groups is explained.

Applications include new rigorous proofs of some folklore results around $\pi_1^{\ets}(\et{X}, x)$, a description of
Grothendieck's short exact sequence for Galois descent in terms of pointed torsor trivializations, and a new \'etale
van Kampen theorem which gives a simple statement about a pushout square of pro-groups that works for covering
families which do not necessarily consist exclusively of monomorphisms. A corresponding van Kampen result for
Grothendieck's profinite groups $\pi_1^{\Gals}$ immediately follows.

\medbreak
\medbreak

\end{abstract}

\section{Introduction}

The \'etale fundamental group $\pi_1(X, x)$ of a pointed connected locally noetherian scheme was defined by
Artin-Mazur in \cite{AM} by means of the cofiltered category of pointed representable hypercovers of $X$ and pointed
simplicial homotopy classes of maps between them. The significance of this object may be seen immediately in the case
$X = \Spec{k}$ where $k$ is a field: fixing a geometric point $x: \Spec{\Omega} \to \Spec{k}$ associated to some
separable closure $\Omega/k$, one may directly compute that $\pi_1^{\ets}(\Spec{k}, x) \cong \Gals(\Omega/k)$, the
absolute Galois group of $k$ with respect to $\Omega$. As the definition of the term ``hypercover'' is independent of
the underlying Grothendieck topology, one may generalize the definition of $\pi_1^{\ets}$ to apply to the hypercovers
of any pointed connected small Grothendieck site $(\mc{C}, x)$, where the ``point'' $x$ is interpreted as a geometric
morphism
\begin{displaymath}
x: \Set \to \Shv{\mc{C}}
\end{displaymath}
of toposes. The object of this paper is to explain some of the basic properties of $\pi_1^{\ets}(\mc{C}, x)$ at this
level of generality, including its relationship with Grothendieck's Galois theory and discrete group torsors and their
trivializations, and to prove a new van Kampen theorem for $\pi_1^{\ets}(\mc{C}, x)$ which simplifies and extends
previous work (in particular \cite{Stix1} and \cite{Z1}) in a new homotopy theoretic direction.

The principal technical device at work here is that of groupoids $\Htor_x$ of \emph{pointed $H$-torsors} for constant
group sheaves $H \deq \Gamma^\ast H$ associated to discrete groups $H$. The characterization of these groupoids as
homotopy fibres by Jardine in \cite{Jardine14} allows one to give precise proofs of facts that previously (to the
author's knowledge) had the status of folklore: in particular, Theorem \ref{charthm} here shows that the pro-groupe
fondamental \'elargi $G$ associated to the full subtopos $\SLC(\mc{C})$ of sums of locally constant objects of
$\Shv{\mc{C}}$ is pro-isomorphic to $\pi_1^{\ets}(\mc{C}, x)$. The subtlety here lies in the identification
\begin{displaymath}
\pi_1^{\ets}(\mc{C}, x) \cong \pi_1^{\ets}(\SLC(\mc{C}), x),
\end{displaymath}
as the latter pro-group is easily shown (Proposition \ref{slcworks}) to be pro-isomorphic to $G$. Artin-Mazur
apparently thought they proved this in \S$10$ of \cite{AM}, but in fact those methods only achieve the
identification of $\pi_1^{\ets}(\mc{C}, x)$ as a representing pro-object in $\pi \Grp$, the category of groups and
homotopy classes of homomorphisms between them, rather than $\Grp$ itself. The proof here is new and requires the
homotopy theoretic characterization of pointed torsors. Pointed torsors may also be used to rigorously establish the
bit of folklore that the profinite completion $\pi_1^{\ets}(\et{X}, x)\,\widehat{\,}\,$ of the nonfinite \'etale
$\pi_1^{\ets}$ of a connected locally noetherian scheme or DM stack is pro-isomorphic to Grothendieck's profinite
fundamental group $\pi_1^{\Gals}(X, x)$ associated to the \emph{finite} \'etale site $\finet{X}$ based at $x$
(Proposition \ref{grothcmp}).

With these issues out of the way, one may study trivializations of pointed torsors and use them to show how
Grothendieck's short exact sequence for Galois descent 
\begin{displaymath}
1 \to \pi_1^{\Gals}(\mc{C}/(Y, y), y) \to \pi_1^{\Gals}(\mc{C}, x) \epi G_Y \to 1
\end{displaymath}
associated to connected Galois objects $(Y, y)$ in pointed Galois categories $(\mc{C}, x, F_x)$ arise naturally from
short exact sequences in pointed nonabelian $H^1_x(\mc{C}, -)$; see Corollary \ref{corgal}. 

The final section of this paper uses the general methods established above to prove a new variant (Corollary
\ref{corevc}) of \'etale van Kampen theorem which is both simple to state (it is just a pushout as in the usual
topological van Kampen theorem) and does \emph{not} require the covering family to consist exclusively of
monomorphisms (as in the case of coverings by open subschemes or substacks; cf. \cite{Z1}). Corollary \ref{corfin}
shows in particular how this result specializes to a statement about Grothendieck's profinite fundamental groups.
The methods employed to prove these statements are conceptual, homotopy theoretic, and in fact give the presumably
stronger Theorem \ref{mainthm} whose statement does not in any direct way depend upon the underlying topology.

\section{Torsors and (geometrically) pointed torsors}

\'Etale homotopy theory begins with the observation that, for suitably nice Grothendieck sites $\mc{C}$, the canonical
constant sheaf functor $\Gamma^\ast: \Set \to \Shv{\mc{C}}$ has a left adjoint $\ccomp: \Shv{\mc{C}} \to \Set$ called
the \emph{connected components} functor. Naturally, this functor has the geometric interpretation of sending a scheme to
its set of scheme-theoretic connected components whenever one is working with some good enough site of schemes with a
subcanonical topology (i.e. all representable presheaves are sheaves).

\subsection{Connectedness and local connectedness}

Recall that a sheaf $F$ on a site $\mc{C}$ is called \emph{connected} if whenever there is a coproduct decomposition
\begin{displaymath}
F = F_1 \sqcup F_2
\end{displaymath}
either $F_1 = \emptyset$ or $F_2 = \emptyset$, where $\emptyset$ denotes the initial sheaf on $\mc{C}$. A site
$\mc{C}$ is called \emph{locally connected} if every sheaf on $\mc{C}$ splits uniquely (up to canonical isomorphism)
as a coproduct of connected sheaves, and if representable sheaves similarly decompose as coproducts of connected
representable sheaves. On the sites of interest here any connected scheme will represent a connected representable
sheaf (see e.g. \cite{Z2}, Lemma $3.3$). It is known that the \'etale sites of locally noetherian Deligne-Mumford
stacks are locally connected ($3.1$, \cite{Z2}). Under these conditions, the aforementioned connected components
functor $\ccomp$ exists (defined by sending each sheaf to its set of connected components) and is easily shown to be
left adjoint to $\Gamma^\ast$. A Grothendieck site $\mc{C}$ with terminal sheaf $\ast$ will be called \emph{connected}
if $\ast$ is connected. However the functor $\ccomp$ may arise, the results below depend only upon its \emph{existence}.

\subsection{Closed model structures and hypercovers}

Say that a category $\mc{C}$ is \emph{small} if its class of morphisms $\Mor{\mc{C}}$ forms a set. Any small
Grothendieck site $\mc{C}$ admits a closed model structure on the associated category of simplicial (pre)sheaves where
the cofibrations are monomorphisms, the weak equivalences are the local weak equivalences, and the fibrations are what
will be called here the \emph{global fibrations}. These closed model structures are due to Joyal in the sheaf case
(\cite{Joyal2}) and Jardine in the presheaf case (\cite{Jardine3}); the reader is encouraged to refer to these papers
for definitions of the terms not defined here. These are known as the ``injective'' model structures, and will
sometimes be used in what follows.

A morphism $f: X \to Y$ of simplicial (pre)sheaves will be called a \emph{local fibration} if it has the local right
lifting property with respect to all the standard inclusions $\khorn \hra \nsimp$ of $k$-horns into the standard
$n$-simplices for $n \geq 0$ (cf. \S$1$ of \cite{Jardine9} for a definition and discussion of the local right lifting
property). A morphism $f: X \to Y$ of simplicial (pre)sheaves will be called a \emph{local trivial fibration} if it is
simultaneously a local fibration and a local weak equivalence; by a theorem of Jardine ($1.12$, \cite{Jardine3}) these
are exactly the morphisms of simplicial (pre)sheaves having the local right lifting property with respect to the
standard inclusions $\dnsimp \hra \nsimp$ of boundaries of the standard $n$-simplices for $n \geq 0$. By a simple
adjointness argument beginning with the observation that $\dnsimp \cong \sk_{n-1}{\nsimp}$, one sees that this local
lifting property is equivalent to the assertion that the (pre)sheaf morphisms
\begin{eqnarray*}
X_0 & \to & Y_0 \\
X_n & \to & \cosk_{n-1}{X_n} \times_{\cosk_{n-1}{Y_n}} Y_n
\end{eqnarray*}
are local epimorphisms for $n \geq 1$. This is true in particular when $Y = K(Z, 0)$, the constant (or ``discrete'')
simplicial (pre)sheaf associated to a (pre)sheaf $Z$. When $Z$ is a scheme and $X$ is representable by a simplicial
scheme this amounts to the classical definition of a hypercover $f: X \to Z$ (cf. \cite{AM}; these observations
appeared in \cite{Jardine8}). For this reason and others it is now standard to call any local trivial fibration of
simplicial (pre)sheaves on a small Grothendieck site $\mc{C}$ a \emph{hypercover}. This is what is meant by the
term``hypercover'' in the remainder of this paper.

\subsection{Torsors and homotopy theory}

Recall that a \emph{torsor} $X$ for a sheaf of groups $G$ on a small Grothendieck site $\mc{C}$ is a sheaf $X$ with an
$G$-action such that there is a sheaf epi $U \epi \ast$ to the terminal sheaf $\ast$ of $\mc{C}$ and a $G$-equivariant
sheaf isomorphism
\begin{displaymath}
X \times U \cong G \times U
\end{displaymath}
called a \emph{trivializaton} of $X$ along $U$. This implies that the sheaf-theoretic quotient $X/G$ is isomorphic to
the terminal sheaf $\ast$. The simplicial sheaf $EG \times_G X$ is defined in sections $U$ as the nerves of the
translation categories $E_{G(U)}(X(U))$ for the actions of $G(U)$ on $X(U)$; each such category is a groupoid so
$\shpi_n(EG \times_G X) \cong 0$ for $n \geq 2$ and, as the isotropy groups locally vanish and the action is locally
transitive there is a local weak equivalence $EG \times_G X \simeq \ast$ (here $\shpi_n$ denotes the sheaf of homotopy
groups in degree $n$; for the definition of a translation category see $1.8$, IV, \cite{GJ}). As there is an isomorphism
of sheaves $\shpi_0(EG \times_G X) \cong X/G$ one sees that $X$ is a $G$-torsor if and only if the map
\begin{displaymath}
EG \times_G X \to \ast
\end{displaymath}
of simplicial sheaves is a local weak equivalence (this is another observation of Jardine; cf. $2.1$, \cite{Jardine4}).
The maps $X \to Y$ of $G$-torsors are $G$-equivariant maps of sheaves, induced as fibres of comparisons of local
fibrations $EG \times_G X \to BG$ (resp. for $Y$), hence are local weak equivalences of constant simplicial sheaves, so
are isomorphisms (following notes of Jardine). The category of $G$-torsors on $\mc{C}$ is therefore a groupoid denoted
$\Gtor(\mc{C})$; its path component set is denoted $H^1(\mc{C}, G)$, and this is the definition of nonabelian $H^1$ of
$\mc{C}$ with coefficients in $G$. This set is pointed by the isomorphism class of the trivial $G$-torsor represented
by $G$ itself.

It has been known at least since \cite{AM} appeared that the \'etale fundamental group $\pi_1^{\ets}(X, x)$ based at
some geometric point $x$ determines $H^1(\mc{C}, H)$ for constant sheaves of discrete groups $H$ where $\mc{C} \deq
\et{X}$, the \'etale site of a connected locally noetherian scheme $X$ pointed by $x$ (i.e. where the covering
families are taken to be surjective sums of \'etale morphisms; here ``\'etale'' is \emph{not} taken to include
``finite''). The following is a quick homotopy-theoretic argument to establish this. The author based it on an earlier
argument of Jardine made in the setting of a Galois category in the sense of (V, \cite{SGA1}). First a lemma:

\begin{lem} \label{connlem}
For any connected small Grothendieck site $\mc{C}$ and any hypercover $U \to \ast$ of the terminal sheaf, one has
\begin{displaymath}
\pi_0(\ccomp U) \cong \ast.
\end{displaymath}
\end{lem}
\begin{proof}
The canonical map $U \to \ast$ induces a map backwards
\begin{displaymath}
[\ast, K(\Gamma^\ast S, 0)] \to [U, K(\Gamma^\ast S, 0)]
\end{displaymath}
in the homotopy category for any set $S$, which is an isomorphism since $U \to \ast$ is a hypercover. On the other hand
$U$ and $\ast$ are cofibrant and $K(\Gamma^\ast S, 0)$ is globally fibrant, so these determine isomorphisms
\begin{displaymath}
\pi(\ast, K(\Gamma^\ast S, 0)) \cong \pi(U, K(\Gamma^\ast S, 0)).
\end{displaymath}
where $\pi$ denotes taking simplicial homotopy classes of maps. One has $\ccomp(\ast) = \ast$ since $\mc{C}$ is
connected, but then by adjunction $\pi(\ast, K(S, 0)) \cong \pi(\ccomp U, K(S, 0))$ so $\Hom(\ast, S) \cong
\Hom(\pi_0(\ccomp U), S)$ for any set $S$. Setting $S = \{0, 1\}$ completes the argument.
\end{proof}
One may paraphrase this by saying that any hypercover of the terminal sheaf on a connected site is automatically
path-connected.

\begin{prop} \label{etfund}
For any connected pointed small Grothendieck site $\mc{C}$ there are bijections
\begin{displaymath}
\pi_{\mathrm{cts}}(\pi_1^{\ets}(\mc{C}), H) \cong H^1(\mc{C}, H)
\end{displaymath}
natural in discrete groups $H$, where $\pi_1^{\ets}(\mc{C})$ is the \'etale fundamental group \`a la Artin-Mazur
constructed by means of pointed (not necessarily representable) hypercovers and $\pi_{\mathrm{cts}}$ is defined as in
the proof.
\end{prop}
\begin{proof}
There is a sequence of identifications
\begin{eqnarray*}
\pi_{\mathrm{cts}}(\pi_1^{\ets}(\mc{C}), H) & \deq  & \colim_{(U, u) \in \HR_\ast(\mc{C})}{\pi(\pi_1(\ccomp U, u), H)} \\
                                            & \cong & \colim_{(U, u) \in \HR_\ast(\mc{C})}{\pi(\ccomp U, BH)} \\
                                            & \cong & \colim_{(U, u) \in \HR_\ast(\mc{C})}{\pi(U, B\Gamma^\ast H)} \\
                                            & \cong & [\ast, B\Gamma^\ast H] \\
                                            & \cong & \pi_0(\Htor(\mc{C})) \deq H^1(\mc{C}, H)
\end{eqnarray*}
where the transition from simplicial homotopy classes of maps to homotopy classes of maps is the generalized Verdier
hypercovering theorem applied to the locally fibrant objects $U$ and $B\Gamma^\ast H$ (cf. \cite{Jardine13}),
$\HR_\ast(\mc{C})$ is the category of (geometrically) pointed hypercovers of the terminal sheaf (implicitly over the
basepoint of $\mc{C}$) and pointed simplicial homotopies between them, $B$ denotes the sectionwise nerve functor, and
$\Htor(\mc{C})$ denotes the groupoid of $H$-torsors associated to the constant sheaf of groups $\Gamma^\ast H$.
\end{proof}
Lemma \ref{connlem} was used here in a subtle way to conflate the fundamental groupoid of $\ccomp U$ with the fundamental
group of $(\ccomp U, u)$, as they are homotopy equivalent (in the sense of taking their nerves) in this case. This argument
applies to \emph{any} topology, not just the \'etale, and so is valid for the analogues of $\pi_1^{\ets}$ in any other
setting where one can talk about hypercovers in a pointed connected site.

The converse question of whether the pointed sets $H^1(\mc{C}, H)$ together with their naturality in discrete groups $H$
together determine $\pi_1^{\ets}(\mc{C})$ is more difficult to answer. Proposition \ref{etfund} may be interpreted as
establishing that the functor
\begin{displaymath}
H^1(\mc{C}, -): \pi \Grp \to \Set
\end{displaymath}
is pro-representable by $\pi_1^{\ets}(\mc{C})$, where $\pi \Grp$ is the category of groups and \emph{simplicial
homotopy classes} of homomorphisms between them (known elsewhere as homomorphisms up to conjugacy or ``exterior''
homomorphisms). As representing pro-objects are unique up to pro-isomorphism ($8.2.4.8$, I, \cite{SGA4}), one knows
that $H^1(\mc{C}, -)$ determines $\pi_1^{\ets}(\mc{C})$ up to pro-isomorphism in $\pi \Grp$, but one does \emph{not}
know at this stage of the argument that it is determined up to pro-isomorphism in $\Grp$. In the case that $H$ need
only vary over \emph{abelian} groups this problem does not arise since the conjugacy relation degenerates.
\emph{Geometrically pointed torsors} (or just ``pointed torsors'' for short) were introduced precisely to fix this
problem wherever Galois theory (in the modern topos theoretic sense) is available.

\section{Pointed torsors and Galois theory}

\subsection{Pointed torsors}

Geometrically pointed torsors may be motiviated by the following simple example from the Galois theory of fields: fix
a base field $k$ and choose a separably closed extension $x: k \hra \Omega$. This determines a geometric point
\begin{displaymath}
x: \Spec{\Omega} \to \Spec{k}
\end{displaymath}
of the associated scheme $\Spec{k}$. Considering any Galois extension $f: k \hra L$ of $k$, the possible lifts $y$ (i.e.
geometric points) making the associated diagram
\begin{displaymath}
\xymatrix{       &  \Spec{L} \ar[d]^f \\
           \Spec{\Omega} \ar[ur]^y \ar[r]_x  &  \Spec{k} }
\end{displaymath}
commute coincide with the possible embeddings $L \hra \Omega$ over $k$, which are of course permuted by the action of
the associated Galois group $\Gal{L/k}$. The scheme $\Spec{L}$ represents a $\Gal{L/k}$-torsor trivializing over
itself; however, if one requires that the Galois action respect a fixed choice of lift (i.e. one makes a fixed choice
of geometric point $y: \Spec{\Omega} \to \Spec{L}$ over $x$) then only the trivial action is possible, as the
associated embedding has been fixed. This property of having no nontrivial pointed automorphisms means that groupoids
of such pointed torsors over $k$ and pointed equivariant maps between them are particularly simple from a homotopy
theoretic point of view: they are disjoint unions of contractible path components.

Suppose $\mc{C}$ is some site of schemes having finite limits and finite coproducts which contains a separably closed
field $\Omega \deq \Spec{\Omega}$. So long as the topology is defined by families of maps that are stable under
pullback, Verdier's criterion from \cite{SGA4}, Expos\'e III applies and determines a site morphism
\begin{displaymath}
x: \mc{C}/\Omega \to \mc{C}/X
\end{displaymath}
associated to any map $x: \Spec{\Omega} \to X$ to some object $X$. If $\mc{C}$ is a site with the \'etale topology then
this further determines a geometric morphism
\begin{displaymath}
x^\ast: \Shv{\et{X}} \leftrightarrows \Shv{\et{\Omega}} \simeq \Set: x_\ast
\end{displaymath}
where one recognizes $x^\ast$ as the functor sending any \'etale sheaf on $X$ to its stalk at the geometric point $x$. 
One readily verifies that the choice of a global section $u: \ast \to x^\ast(U)$ associated to a sheaf $U$ represented
by an object $U \to X$ of $\et{X}$ is equivalent to the choice of a geometric point $u: \Spec{\Omega} \to U$ over $x$. 

More generally, suppose that
\begin{displaymath}
x^\ast: \Shv{\mc{C}} \leftrightarrows \Set: x_\ast 
\end{displaymath}
is a geometric morphism with direct image functor $x_\ast$ for some small Grothendieck site $\mc{C}$; such a pair
$(\mc{C}, x)$ will be called a \emph{pointed site} with basepoint $x$. The functor $x^\ast$ is an inverse image functor
so is exact and thus preserves local weak equivalences: by exactness $x^\ast$ commutes with Kan's $\Exinfn$ functor and
thus it suffices to prove the statement for locally fibrant objects, but then one may factor the map as a local weak
equivalence right inverse to a hypercover followed by a hypercover (cf. e.g. $8.4$, II, \cite{GJ}), and one may directly
check that $x^\ast$ preserves hypercovers by using their definition in terms of coskeleta. As $x^\ast$ is exact it also
preserves the Borel construction $EH \times_H X$ associated to any $H$-torsor $X$ for any sheaf of groups $H$ on
$\mc{C}$, so in summary it preserves the local weak equivalence
\begin{displaymath}
EH \times_H X \xrightarrow{\simeq} \ast
\end{displaymath}
and thus sends $H$-torsors on $\mc{C}$ to $x^\ast H$-torsors in $\Set$.

A \emph{geometrically pointed $H$-torsor} on the pointed site $(\mc{C}, x)$ for a sheaf of groups $H$ is an $H$-torsor
$Y$ on $\mc{C}$ together with a global section $y: \ast \to x^\ast(Y)$. A \emph{morphism} of (geometrically) pointed
$H$-torsors is a morphism of the underlying $H$-torsors that respects the points. Lemma $1$ of \cite{Jardine14}
establishes in particular that the groupoid $\Htor(\mc{C})_x$ of pointed $H$-torsors on $\mc{C}$ is the homotopy fibre
of the map
\begin{displaymath}
\Htor(\mc{C}) \to x^\ast(H)\textrm{-}\Tors(\Set).
\end{displaymath}
Suppose now that $H$ is a \emph{constant} sheaf of groups. Then every pointed $H$-torsor on $\mc{C}$ is locally
constant on $\mc{C}$ so the groupoids $\Htor(\mc{C})_x$ for variable discrete groups $H$ all belong to the subtopos
$\SLC(\mc{C}) \subset \Shv{\mc{C}}$ of sums of locally constant objects of $\Shv{\mc{C}}$. The inverse image $x^\ast:
\Shv{\mc{C}} \to \Set$ then restricts (cf. \cite{Mord2}) to an inverse image
\begin{displaymath}
x^\ast: \SLC(\mc{C}) \to \Set.
\end{displaymath}

\subsection{Digression on Galois theory}

Now assume that $(\mc{C}, x)$ is also connected. In \cite{Mord1}, Moerdijk showed that there are then pointed topos
equivalences
\begin{displaymath}
\SLC(\mc{C}) \simeq \Shv{\mathbf{G/U}} \simeq \mc{B}G
\end{displaymath}
analogous to that of classical Galois theory, where $G$ is a prodiscrete localic group (pro-object in the category of
localic groups, which themselves are group objects in the category of locales). Here one may conflate the diagram $G$
with its limit object $\lim_{}G$ in localic groups (in a sense this is the reason for considering localic rather than
topological groups in this context). The ``prodiscrete'' adjective means that $G$ may be interpeted as a cofiltered
diagram of discrete localic groups, just as classical Galois theory deals with limits of discrete topological groups.
The site $\mathbf{G/U}$ is that of the right cosets $G/U$ for localic open subgroups $U$ of $G$ (these are discrete),
with the obvious $G$-action, and all $G$-equivariant maps between them as $G$-sets, where a $G$-set $Z$ is a set $Z$
with a $G$-action in the localic sense. The topos $\mc{B}G$ is that of all discrete $G$-sets, and in particular the
latter equivalence asserts that any discrete $G$-set of the form $G/U$ for a localic open subgroup $U$ represents a
sheaf on the site $\mathbf{G/U}$, so the topology is subcanonical. It is known that any prodiscrete localic group
corresponds to a pro-group with surjective transition maps, and the canonical maps $G \to G_i$ are all also surjective
($1.4$, \cite{Mord1}).

Under Moerdijk's equivalence the inverse image $x^\ast: \SLC(\mc{C}) \to \Set$ goes to the functor $g^\ast: \mc{B}G \to
\Set$ that forgets the $G$-action \cite{Mord2}; this functor obviously reflects epis so it is faithful and thus the
topos $\mc{B}G$ has enough points, namely $\{g\}$. This point $g$ restricts to a forgetful functor
\begin{displaymath}
u: G/U \to \Set.
\end{displaymath}
Recall that if $\Set$ is equipped with the standard topology where the covering families are surjections then there is
an equivalence $\Shv{\Set} \simeq \Set$ determined by $F \mapsto F(\ast)$ on the one hand and $X \mapsto \Hom(-, X)$ on
the other. Following \cite{Jardine14}, the direct image $u_\ast \simeq g_\ast: \Set \to \Shv{G/U}$ sends any set $X$ to
the sheaf $\Hom(u(-), X)$ on $G/U$, and the left adjoint $u^\ast \simeq g^\ast: \Shv{G/U} \to \Set$ is the left Kan
extension defined by
\begin{displaymath}
u^\ast(F)(\ast) \deq \colim_{\ast/X}{F(X)}
\end{displaymath}
where the index category $\ast/X$ has maps $\ast \to X$ as objects where the $X$ are sets of the form $G/U$ for open
localic subgroups $U$ and morphisms are commutative triangles over morphisms in the category $G/U$. The identity
elements $e: \ast \to G_i$ represent the pro-group $G$ itself in this index category, and this subcategory is cofinal
as usual so one arrives at the useful characterization
\begin{displaymath}
u^\ast(F) = \colim_{i}{F(G_i)}.
\end{displaymath}
The pointed topos $(\Shv{\mathbf{G/U}}, u)$ has enough points since it is pointed equivalent to $(\mc{B}G, g)$, and in
fact the list $\{u\}$ suffices since equivalences of categories are faithful functors.

A word of explanation: intuitively, the subtopos $\SLC(\mc{C})$ captures and isolates the covering space theory of
$\Shv{\mc{C}}$. The category of locally constant \emph{finite} sheaves on the \'etale site of a scheme $X$ is known (by
descent theory) to be equivalent to the \emph{finite} \'etale site of $X$. More generally, any locally constant sheaf
of sets on $\et{X}$ for a scheme $X$ is represented by a (not necessarily finite) \'etale map (cf. $2.2$, Expos\'e IX,
III, \cite{SGA4}), so the connection with geometry is closer than one might expect a priori.

\subsection{Characterization of $\pi_1^{\ets}$} \label{fusection}

Under Moerdijk's equivalence the groupoid $\Htor(\mc{C})_x$ of pointed $H$-torsors over $x$ for any \emph{constant} sheaf
of groups $H$ on a connected site $\mc{C}$ goes to a groupoid $\Htor(\mathbf{G/U})_u$ of pointed $H$-torsors over $u$.
Equivalences of categories preserve path components and pointed torsors over $x$ and $u$ have no nontrivial (pointed)
automorphisms since $x^\ast$ and $u^\ast$ are faithful on $\SLC(\mc{C})$, so there are induced weak equivalences of
groupoids
\begin{displaymath}
\Htor(\mc{C})_x \simeq \Htor(\mathbf{G/U})_u
\end{displaymath}
natural in $H$. It therefore suffices to study the latter class of groupoids for the purpose of computing
$\pi_0(\Htor(\mc{C})_x)$. By Lemma $1$ of \cite{Jardine14} the groupoid $\Htor(\mathbf{G/U})_u$ is the homotopy fibre of
the canonical map
\begin{displaymath}
u^\ast: \Htor(\mathbf{G/U}) \to \Htor(\Set)
\end{displaymath}
The analogue of Corollary $11$ of \cite{Jardine14} in this context determines bijections
\begin{displaymath}
\pi_0(\Htor(\mc{C})_x) \cong \pi_0(\Htor(\mathbf{G/U})_u) \cong \colim_{i \in I}{\Hom(\check{C}(G_i), BH)}
\end{displaymath}
natural in $H$ where $I$ is the indexing category for the prodiscrete localic group $G$ and $\check{C}(G_i)$ is the
simplicial \v{C}ech resolution associated to the $G$-equivariant epi $G_i \epi \ast$. As $H$ is a constant sheaf of
groups (rather than groupoids), Example $13$ of \cite{Jardine14} applies to give bijections
\begin{displaymath}
H^1_x(\mc{C}, H) \deq \pi_0(\Htor(\mc{C})_x) \cong \colim_i{\Hom(G_i, H)}
\end{displaymath}
natural in $H$ and the latter may be further identified with the set $\Hom_{\mathrm{\loc}}(G, H)$ of maps of prodiscrete
localic groups from $G$ to $H$ to make complete the analogy with the case of ordinary torsors discussed above. 

The upshot of all this is summarized in
\begin{prop} \label{prorep}
Suppose $(\mc{C}, x)$ is a connected (thus locally connected) small Grothendieck site pointed by a choice of geometric
morphism $x: \Set \to \Shv{\mc{C}}$ (``point'' in topos language). Then the pro-group $G$ associated to the subtopos
$\SLC(\mc{C})$ of $\Shv{\mc{C}}$ is determined up to pro-isomorphism by the pointed nonabelian $H^1_x(\mc{C}, -)$
functor associated to $x$.
\end{prop}
\begin{proof}
After the above discussion one must only note that $H^1_x(\mc{C}, -)$ is pro-representable by $G$ and that representing
pro-objects are unique up to pro-isomorphism.
\end{proof}

The pro-group $G$ here is exactly what Grothendieck called the ``pro-groupe fondamental \'elargi'' based at $x$ (pg.
$110$, Book $2$, \cite{SGA3}). The expected explicit identification of $G$ is given by $G \cong \pi_1^{\ets}(\mc{C},
x)$ where the latter group is the fundamental pro-group of Artin-Mazur defined by means of pointed representable
hypercovers. The standard reference for this identification is \S$10$ of \cite{AM}, but unfortunately the argument
there boils down to the characterization of \emph{un}pointed nonabelian $H^1(|\ccomp K|, H)$ for hypercovers $K$ in
classical topological covering space theory, so it cannot be considered to give a complete characterization of $G$. In
other words, it runs into the same ``homomorphisms up to simplicial homotopy'' versus ``actual homomorphisms'' problem
mentioned after Proposition \ref{etfund}.

The statement $\pi_1^{\ets}(\et{X}, x) \cong G$ for small \'etale sites $\et{X}$ of connected locally noetherian
schemes or DM stacks seems to have also been considered in (pg. $111$, Book $2$, \cite{SGA3}). There Grothendieck is
considering \emph{pointed} $H$-torsors for constant group sheaves $H$, and his argument reduces to the explicit
description of $\Hom(G, H) \cong H^1_x(\mc{C}, H)$ in terms of pointed descent data associated to \v{C}ech resolutions
$\check{C}(U)$ for the various \'etale coverings $U \epi \ast$. Unfortunately he gives no proof of the crucial last
lines of his argument, so to the author's knowledge there is a gap in the literature. One may fill this gap as follows.

Fix a pointed small Grothendieck site $(\mc{C}, x)$, a constant sheaf of groups $H$ on $\mc{C}$ with stalk also denoted
$H$, and a pointed representable sheaf epi $(U, u) \epi \ast$ of $\Shv{\mc{C}}$. Let $?_{|U}: \Shv{\mc{C}} \to
\Shv{\mc{C}/(U, u)}$ denote the corresponding restriction functor; this functor is exact (right adjoint of a topos
morphism and it preserves sheaf epis), so it sends pointed $H$-torsors on $\mc{C}$ (with respect to $x$) to pointed
$H$-torsors (with respect to $u$) on $\mc{C}/(U, u)$, and morphisms of pointed $H$-torsors on $\mc{C}$ to morphisms of
pointed $H$-torsors on $\mc{C}/(U, u)$. Let $F_U$ denote the homotopy fibre of the induced map
\begin{displaymath}
\Htor_x \xrightarrow{?_{|U}} \Htor_u
\end{displaymath}
of groupoids where $\Htor_x$ is the groupoid of pointed $H$-torsors on $\mc{C}$ with respect to $x$ and $\Htor_u$ is
the groupoid of pointed $H$-torsors on $\mc{C}/(U, u)$ with respect to $u$. Then the objects of $F_U$ are morphisms $H
\to T_{|U}$ of \emph{pointed} $H$-torsors (the trivial $H$-torsor $H$ being pointed by $e \in H$) and the morphisms are
commutative triangles
\begin{displaymath}
\xymatrix{   &  T_{|U} \ar[d]^{m_{|U}} \\
            H \ar[ur] \ar[r] &  T^\prime_{|U} }
\end{displaymath}
of morphisms of pointed $H$-torsors on $\mc{C}/(U, u)$ where $m: T \to T^\prime$ is a morphism of pointed $H$-torsors
on $\mc{C}$. The objects of $F_U$ correspond exactly to pointed trivializations $\sigma: (U, u) \to T$ so that $F_U$ is
equivalent to the groupoid $F_U$ whose objects are pointed trivializations of the form $\sigma$ and whose morphisms are
commutative triangles
\begin{displaymath}
\xymatrix{      &  T \ar[d]^m \\
           (U, u) \ar[r]^(0.6){\sigma^\prime} \ar[ur]^{\sigma} & T^\prime }
\end{displaymath}
of pointed maps in $\mc{C}$ where $m$ is a morphism of pointed $H$-torsors. The map $F_U \to \Htor_x$ is that which
forgets the pointed trivializations.

\begin{lem} \label{cident}
With the definitions above, there are bijections
\begin{displaymath}
\pi_0(F_U) \cong \Hom(\check{C}(U, u), BH)
\end{displaymath}
natural in discrete groups $H$, where $\check{C}(U, u)$ is the pointed \v{C}ech resolution associated to the pointed
sheaf epi $(U, u) \epi \ast$ of $\mc{C}$.
\end{lem}
\begin{proof}
By Lemma $2$ of \cite{Jardine14} the path components of $F_U$ correspond to the path components of the cocycle
category $h_{\check{C}}(\ast, BH)_x$ of those pointed \v{C}ech cocycles under $(U, u)$. Each pointed trivialization
$\sigma: (U, u) \to T$ determines a pointed \v{C}ech cocycle
\begin{displaymath}
\ast \xleftarrow{\simeq} \check{C}(U, u) \xrightarrow{\sigma_\ast} \check{C}(T) \to BH.
\end{displaymath}
Such cocycles are initial in their respective path components of $h_{\check{C}}(\ast, BH)_x$, so these path
components correspond to certain maps $\check{C}(U, u) \to BH$. Any element in $\Hom(\check{C}(U, u), BH)$
determines a pointed $H$-torsor $T$ equipped with a fixed pointed trivialization $\sigma$ by $(U, u)$, so this
correspondence is surjective, hence bijective.
\end{proof}

The category $E_x$ of pointed sheaf epis $(U, u) \epi \ast$ on $\mc{C}$ certainly has products, and for the sake of
torsor trivialization arguments one may assume that parallel maps
\begin{displaymath}
\xymatrix{ (V, v) \ar@<0.5ex>[rr] \ar@<-0.5ex>[rr] \ar@{->>}[dr] &  &  (U, u) \ar@{->>}[dl] \\
           & \ast & }
\end{displaymath}
are identified so that $E_x$ is (co)filtered.

Say that a pointed $H$-torsor $(T, t)$ \emph{admits a pointed trivialization} if there exists a pointed representable
sheaf epi $(U, u) \epi \ast$ and a pointed section $\sigma: (U, u) \to T$. Pointed trivializations exist in most cases
of interest:

\begin{lem} \label{torstriv}
Suppose $(\mc{C}, x)$ is a pointed small Grothendieck site with a subcanonical topology (i.e. representable presheaves
are sheaves) and arbitrary small coproducts. Then any pointed $H$-torsor for a sheaf of groups $H$ on $\mc{C}$ admits
a pointed trivialization by a representable sheaf epi.
\end{lem}
\begin{proof}
Any pointed $H$-torsor $(T, t)$ admits some ordinary trivialization $\sigma: V \to T$ from a sheaf epi $V \epi \ast$,
but then the composite $V \to T \to \ast$ is an epi so that the canonical map $T \to \ast$ is also an epi. As a sheaf,
$T$ is a colimit of representable presheaves $U_i$ for some small index category $I$, which are sheaves since the
topology is subcanonical, and there is a sheaf epi $U \deq \sqcup_{i \in I}{U_i} \epi T$ so that the composite $U \to
T \to \ast$ is a sheaf epi. In particular the map $U \epi T$ is a representable trivialization of $T$, and as it is a
sheaf epi, any point $t \in x^\ast(T)$ lifts to some point $u \in x^\ast(U)$, so that $(U, u) \epi T$ is a pointed
trivialization of $T$.
\end{proof}

This is true in particular for pointed sites of connected locally noetherian schemes and DM stacks with nonfinite
\'etale topologies. The homotopy long exact sequence associated to the fibre sequence
\begin{displaymath}
F_U \to \Htor_x \to \Htor_u
\end{displaymath}
in groupoids for a representable pointed sheaf epi $(U, u) \epi \ast$ has the form
\begin{displaymath}
1 \to \pi_1(\Htor_u) \to \pi_0(F_U) \to \pi_0(\Htor_x) \to \pi_0(\Htor_u).
\end{displaymath}
for any \emph{constant} sheaf of groups $H$ on a pointed connected site $(\mc{C}, x)$, since the objects of $\Htor_x$
have no pointed automorphisms. The trivial pointed torsor $H$ on $\mc{C}/(U, u)$ is represented by $H \times U$ pointed
by $e \times u$ over $x$ with its canonical pointed projection $\pr_u: H \times U \to U$. The sheaf $H \times U$ is
obviously locally constant since $H$ is constant, so it lives in $\SLC(\mc{C})$, but then it has no nontrivial pointed
$(U, u)$-automorphisms as any such must be the identity on $x^\ast$, and $x^\ast$ is faithful for $\SLC(\mc{C})$ by
Galois theory. Thus $\pi_1(\Htor_u)$ vanishes at the basepoint given by the trivial $H$-torsor, so the sequence above
always reduces to an exact sequence
\begin{displaymath}
1 \to \pi_0(F_U) \to \pi_0(\Htor_x) \to \pi_0(\Htor_u).
\end{displaymath}

\begin{lem}
With the definitions above there are bijections
\begin{displaymath}
\colim_{(U, u) \in E_x}{\pi_0(F_U)} \cong \pi_0(\Htor_x) \deq H^1_x(\mc{C}, H)
\end{displaymath}
natural in discrete groups $H$ whenever pointed $H$-torsors on $(\mc{C}, x)$ admit pointed trivializations by
representable sheaf epis for all constant sheaves of discrete groups $H$.
\end{lem}
\begin{proof}
Taking the filtered colimit over $E_x$ of the homotopy short exact sequences of pointed sets above, one obtains a short
exact sequence
\begin{displaymath}
1 \to \colim_{E_x}{\pi_0(F_U)} \to \pi_0(\Htor_x) \to \ast
\end{displaymath}
of pointed sets since one has $\colim_{E_x}{\pi_0(\Htor_u)} \cong \ast$ as every pointed $H$-torsor on any $\mc{C}/(U,
u)$ pointed trivializes on a sufficiently fine choice of pointed sheaf epi $(V, v) \epi \ast$ dominating $(U, u)$. The
middle map is therefore surjective by exactness. The filtered colimit of the maps $\pi_1(\Htor_x) \to \pi_1(\Htor_u)$ for
\emph{any} choice of basepoint of $\Htor_x$ is surjective since all non-trivial torsors eventually pointed trivialize and
one knows that the trivial $H$-torsor on any $(U, u)$ has no nontrivial pointed automorphisms by the above discussion.
Therefore the colimit map $\pi_0(\colim_{E_x}{F_U}) \to \pi_0(\Htor_x)$ must be injective by Lemma \ref{gpdlem} below.
The middle map must therefore be bijective, as was to be shown.
\end{proof}

The following Theorem gives the desired characterization of $G$:
\begin{thm} \label{charthm}
Suppose $(\mc{C}, x)$ is a connected, pointed small Grothendieck site with finite limits and arbitrary small
coproducts where pointed $H$-torsors admit pointed trivializations by representable sheaf epis for all constant
sheaves of discrete groups $H$. Then the pro-groupe fondamental \'elargi $G$ associated to the full subtopos
$\SLC(\mc{C})$ of $\Shv{\mc{C}}$ via Galois theory is pro-isomorphic to the \'etale fundamental pro-group
$\pi_1^{\ets}(\mc{C}, x)$ constructed by means of pointed representable \v{C}ech hypercovers.
\end{thm}
\begin{proof}
The strategy is to demonstrate that $\pi_1^{\ets}(\mc{C}, x)$ pro-represents the functor $H^1_x(\mc{C}, -)$ so that it
must be pro-isomorphic to $G$. By the previous two Lemmas one has identifications
\begin{eqnarray*}
\Hom_{\mathrm{cts}}(\pi_1^{\ets}(\mc{C}, x), H) & \cong & \colim_{E_x}{\Hom(\pi_1(\ccomp \check{C}(U, u)), H)} \\
  & \cong & \colim_{E_x}{\Hom(\check{C}(U, u), B\Gamma^\ast H)} \\
  & \cong & \colim_{E_x}{\pi_0(F_U)} \\
  & \cong & \pi_0(\Htor_x) \deq H^1_x(\mc{C}, H)
\end{eqnarray*}
natural in discrete groups $H$, so the result follows.
\end{proof}

In particular the above result does not assume $\mc{C}$ has the \'etale topology (despite the use of the word
``\'etale'' which is present here for historical reasons) and makes no use of descent theory, although in retrospect
it does show that descent-theoretic methods may be applied to the study of $\pi_1^{\ets}(\mc{C}, x)$ in suitable
geometric settings. The following comparison result shows that it suffices to use sheaf-theoretic hypercovers:

\begin{prop} \label{shvhyp}
Suppose $(\mc{C}, x)$ is a pointed connected small Grothendieck site with a subcanonical topology (i.e. representable
presheaves are sheaves), finite limits, and arbitrary small coproducts. Then there is a pro-isomorphism
\begin{displaymath}
\pi_1^{\ets}(\mc{C}, x) \cong \pi_1^{\ets}(\mc{C}, x)^{\mathrm{rep}}
\end{displaymath}
where the lefthand side is the fundamental pro-group defined by means of pointed hypercovers and the righthand side is
the classical fundamental pro-group defined by means of pointed \emph{representable} hypercovers.
\end{prop}
\begin{proof}
The essential fact here is that any hypercover $U \to \ast$ of the terminal sheaf $\ast$ may be dominated by a
representable hypercover $V \to \ast$ (cf. Lemma $2.2$, \cite{Jardine8}). A careful reading of the argument there shows
that the representable simplicial scheme $X$ at the base may be replaced by the terminal sheaf $\ast$ in the present
context without losing representability of the split hypercover, even when the terminal sheaf $\ast$ itself is not
representable (e.g. on certain fibred sites), as it is a constant simplicial object.

The argument of \cite{Jardine8} is inductive and begins with the observation that any sheaf epi $U \epi \ast$ to the
terminal sheaf may be dominated by a representable sheaf epi: every presheaf (in particular every sheaf $U$) is a
colimit of representable presheaves $U_i$ (which are sheaves since the topology is subcanonical) on some small index
category $I$. This colimit appears as the coequalizer
\begin{displaymath}
\bigsqcup_{k: i \to j \in I}{U_i} \rars \bigsqcup_{j \in I}{U_j} \epi U
\end{displaymath}
so in particular the indicated map is always an epi and thus it determines a commutative triangle
\begin{displaymath}
\xymatrix{     &   U \ar@{->>}[d] \\
           \sqcup_{j \in I}{U_j} \ar@{->>}[ur] \ar@{->>}[r]  &  \ast }
\end{displaymath}
of epis whenever $U \to \ast$ is a sheaf epi. In particular the map
\begin{displaymath}
V \deq \sqcup_{j \in I}{U_j} \epi U
\end{displaymath}
is an epi in this case so that $x^\ast(V) \to x^\ast(U)$ is also an epi by exactness, thus any point $u_x \in
x^\ast(U)$ lifts to a point $v_x \in x^\ast(V)$ and therefore any \emph{pointed} sheaf epi may be dominated by a
\emph{pointed} representable sheaf epi.

Thus any pointed hypercover is dominated by some pointed \emph{representable} hypercover. Now the category
$\HR_\ast(\mc{C})$ of pointed hypercovers of the terminal sheaf of $\mc{C}$ and pointed simplicial homotopy classes of
maps between them has products and equalizers so the result follows by a cofinality argument.
\end{proof}

One may also wonder what happens if $\pi_1^{\ets}(\SLC(\mc{C}), x)$ is computed instead of $\pi_1^{\ets}(\mc{C}, x)$:

\begin{prop} \label{slcworks}
Suppose $(\mc{C}, x)$ is a pointed connected small Grothendieck site having all finite limits and a subcanonical
topology. Then the pro-group $G$ associated to $\SLC(\mc{C}, x)$ is pro-isomorphic to the pro-group
$\pi_1^{\ets}(\SLC(\mc{C}), x)$ defined by means of pointed representable hypercovers in $\SLC(\mc{C})$.
\end{prop}
\begin{proof}
One directly calculates
\begin{eqnarray*}
\Hom_{\mathrm{cts}}(\pi_1^{\ets}(\SLC(\mc{C}), x), H)
& \cong & \colim_{(U, u) \epi \ast}{\Hom(\pi_1(\ccomp \check{C}(U, u)), H)} \\
& \cong & \colim_{(U, u) \epi \ast}{\Hom(\check{C}(U, u), B\Gamma^\ast H)} \\
& \cong & \colim_{(Y, y) \in \Gals(\SLC(\mc{C}))}{\Hom(\check{C}(Y, y), B\Gamma^\ast H)} \\
& \cong & \colim_{(G_Y,\, e_{G_Y})}{\Hom(\check{C}(G_Y), B\Gamma^\ast H)} \\
& \cong & \colim_{G_Y}{\Hom(G_Y, H)} \deq \Hom_{\loc}(G, H)
\end{eqnarray*}
where the $(U, u) \epi \ast$ are pointed representable sheaf epis over $x$, the $(Y, y) \in \SLC(\mc{C})$ are pointed
Galois objects for $(\SLC(\mc{C}), x)$ which are cofinal among such sheaf epis by Galois theory (cf. Prop. $3.1.1$,
Thm. $3.3.8$, \cite{Dubuc1}), the $G_Y \deq \Gals(Y)$ are the associated Galois groups regarded as discrete localic
groups, and $\check{C}(U, u)$ is the \v{C}ech construction for any pointed sheaf epi $(U, u) \epi \ast$. As these
bijections are natural in $H$ the result follows.
\end{proof}
Proposition \ref{slcworks} was previously established by Artin-Mazur (\S$9$, \cite{AM}) who used a slightly different
method.

\subsection{Application to finite \'etale sites}

As a consequence of (V, \cite{SGA1}) and ($4.2$, \cite{Noo1}) one has pointed equivalences
\begin{displaymath}
\Shv{\finet{X}, x} \simeq \Shv{\pi_1^{\Gals}(X, x)\textrm{-}\Set_{df}} \simeq \mc{B}\pi_1^{\Gals}(X, x)
\end{displaymath}
whenever $(\finet{X}, x)$ is the \emph{finite} \'etale site of a connected locally noetherian scheme or DM stack with
geometric point $x$, where $\pi_1^{\Gals}(X, x)$ denotes the profinite fundamental group of $(\finet{X}, x)$ determined
by Grothendieck's general theory of Galois categories. The same techniques as in the proof of Proposition
\ref{slcworks} establish that there is a pro-isomorphism
\begin{displaymath}
\pi_1^{\ets}(\finet{X}, x) \cong \pi_1^{\Gals}(X, x)
\end{displaymath}
in this situation, at least where $\pi_1^{\ets}(\finet{X}, x)$ is defined by pointed \emph{representable} hypercovers,
so that $\pi_1^{\Gals}(X, x)$ pro-represents the functor
\begin{displaymath}
H^1_x(\finet{X}, -): \FinGrp \to \Set
\end{displaymath}
where $\FinGrp$ is the category of finite groups and group homomorphisms between them. Any finite group torsor on
$\et{X}$ automatically belongs to $\finet{X}$, so $H^1_x(\et{X}, -)$ and $H^1_x(\finet{X}, -)$ are isomorphic as
functors on $\FinGrp$. The following Proposition establishes a relationship between the finite and nonfinite \'etale
sites:

\begin{prop} \label{grothcmp}
Suppose $(\et{X}, x)$ is the pointed (nonfinite) \'etale site of a connected locally noetherian scheme or DM stack.
Then the profinite Grothendieck fundamental group $\pi_1^{\Gals}(X, x)$ associated to $\finet{X}$ is pro-isomorphic
to $\pi_1^{\ets}(\et{X}, x)\,\widehat{\,}\,$, the profinite completion of the classical \'etale homotopy pro-group
$\pi_1^{\ets}(\et{X}, x)$ constructed by means of pointed representable hypercovers.
\end{prop}
\begin{proof}
As seen above $\pi_1^{\Gals}(X, x)$ pro-represents the functor $H^1_x(\et{X}, -): \FinGrp \to \Set$, and on the other
hand $\pi_1^{\ets}(\et{X}, x)$ is known to pro-represent $H^1_x(\et{X}, -): \Grp \to \Set$ by Theorem \ref{charthm}.
Therefore $\pi_1^{\ets}(\et{X}, x)\,\widehat{\,}$ also pro-represents $H^1_x(\et{X}, -): \FinGrp \to \Set$, so it must
be pro-isomorphic to $\pi_1^{\Gals}(X, x)$.
\end{proof}

\section{Pullbacks of groupoids and $\pi_0$}

The results below will require some basic facts about the path components functor $\pi_0$ as it pertains to pullbacks
of groupoids. The following Lemma exemplifies the utility of representing a homotopy fibre sequence by a diagram that
commutes on the nose:
\begin{lem} \label{gpdlem}
Suppose $f: G \to H$ is a map of groupoids inducing surjections $\pi_1(G, x) \epi \pi_1(H, f(x))$ on all fundamental
groups, $y$ an object of $H$, and $F_y$ the homotopy fibre of $f$ over $y$. Then the induced map
\begin{displaymath}
\pi_0(f): \pi_0(F_y) \to \pi_0(G)
\end{displaymath}
is an injection.
\end{lem}
\begin{proof}
Form the pullback square
\begin{displaymath}
\xymatrix{ G \times_H H^I \ar[r]^(0.6){f_\ast} \ar[d]^{d_{1\ast}}  &  H^I \ar[d]^{d_1} \\
           G \ar[r]_f \ar@/^2ex/[u]^s  &  H }
\end{displaymath}
and note the section $s$ (thus weak equivalence) of $d_{1\ast}$ defined by sending any object $x$ of $G$ to $1: f(x) \to
f(x)$ in $G \times_H H^I$. Generally, an object of $G \times_H H^I$ is an object $x$ of $G$ together with a morphism
$\omega: f(x) \to z$ of $H$, and a morphism is a pair of morphisms $(m, n)$ making the square
\begin{displaymath}
\xymatrix{ f(x) \ar[r]^{\omega} \ar[d]_{f(m)}  &  z \ar[d]^n \\
           f(x^\prime) \ar[r]^{\omega^\prime}  &  z^\prime }
\end{displaymath}
commute in $H$. Then $\pi = d_0f_\ast: G \times_H H^I \to H$ is a fibration (functorially) replacing $f$ and the pullback
\begin{displaymath}
\xymatrix{ F_y \ar[r] \ar[d]  &  G \times_H H^I \ar[d]^\pi \\
           \ast \ar[r]^y        &  H }
\end{displaymath}
defines a model $F_y$ for the homotopy fibre over $y$. Since $\pi s = f$, the objects of $F_y$ look like $\omega: f(x) \to
y$ and the morphisms are triangles
\begin{displaymath}
\xymatrix{ f(x) \ar[rr]^{f(m)} \ar[dr]_{\omega}  &   &  f(x^\prime) \ar[dl]^{\omega^\prime} \\
            &  y  &  }
\end{displaymath}
that commute on the nose. Suppose $\omega: f(x) \to y$, $\omega^\prime: f(x^\prime) \to y$ are two objects of $F_y$ in the
same path component of $G \times_H H^I$. Then there is a commutative square
\begin{displaymath}
\xymatrix{ f(x) \ar[r]^{\omega} \ar[d]_{f(m)}  &  y \ar[d]^n \\
           f(x^\prime) \ar[r]^{\omega^\prime}  &  y  }
\end{displaymath}
in $H$, but the map $\pi: G \times_H H^I \to H$ is surjective on $\pi_1$ so $n$ has a lift
\begin{displaymath}
\xymatrix{ f(x) \ar[r]^{\omega} \ar[d]_{f(p)}  &  y \ar[d]^n \\
           f(x) \ar[r]^{\omega}                &  y }
\end{displaymath}
where $p$ is an automorphism of $x$ in $G$. But then there is a commutative triangle
\begin{displaymath}
\xymatrix{ f(x) \ar[rr]^{f(mp^{-1})} \ar[dr]_{\omega}  &   &  f(x^\prime) \ar[dl]^{\omega^\prime} \\
           &  y  &  }
\end{displaymath}
so $\omega$ and $\omega^\prime$ are in the same path component of $F_y$.
\end{proof}

To apply the above Lemma to pullbacks of groupoids, it is helpful to know the following slightly unusual
characterization of pullback diagrams in sets:
\begin{lem}
A square diagram of sets
\begin{displaymath}
\xymatrix{ A \ar[r]^{g_\ast} \ar[d]_{f_\ast}  &  B \ar[d]^f \\
           C \ar[r]^g         &  D }
\end{displaymath}
is a pullback if and only if for each $c \in C$ the fibre $A_c$ of $f_\ast$ over $c$ is isomorphic to the fibre
$B_{gc}$ of $f$ over $gc \in D$ via the induced map $i_c: A_c \to B_{gc}$.
\end{lem}
\begin{proof}
Assume the induced map $i_c: A_{c^\prime} \to B_{gc^\prime}$ is an isomorphism for each $c^\prime \in C$ and suppose
$(c, d, b) \in C \times_D B$. Then $g(c) = d = f(b)$, so $b \in B_{gc}$ and there is a unique element $a_b \in A_c$
such that $i_c(a_b) = b$. To see that $(c, d, b) \mapsto a_b \deq i_c^{-1}(b)$ is the desired assignment, observe that
$f_\ast(a_b) = c$ since $a_b \in A_c$ and $g_\ast(a_b) = i_c(a_b) = b$. The collection of such assignments determines a
unique map $C \times_D B \to A$, so $A$ is a pullback of the diagram. The converse is trivial.
\end{proof}

\begin{prop} \label{pbprop}
Suppose $f: G \to H$ is a map of groupoids inducing surjections $\pi_1(G, x) \epi \pi_1(H, f(x))$ on all fundamental
groups. Then any homotopy pullback $P$ along a map $g: N \to H$ of groupoids induces a pullback
\begin{displaymath}
\xymatrix{ \pi_0(P) \ar[r] \ar[d]  &  \pi_0(G) \ar[d] \\
           \pi_0(N) \ar[r]         &  \pi_0(H) }
\end{displaymath}
of path component sets.
\end{prop}
\begin{proof}
Replace $f$ by a fibration. For any object $y$ of $N$, consider the diagram
\begin{displaymath}
\xymatrix{ F_y \ar[r]^{y_\ast} \ar[d]    &  P \ar[r]^{g_\ast} \ar[d]^{f_\ast}  &  G \ar[d]^f \\
           \ast \ar[r]_y        &  N \ar[r]_g       &  H }
\end{displaymath}
of pullbacks and the induced comparison
\begin{displaymath}
\xymatrix{ \pi_0(F_y) \ar[r]^{y_\ast} \ar[d]_{\cong}  &  \pi_0(P) \ar[r]^{f_\ast} \ar[d]^{g_\ast}  &  \pi_0(N) \ar[d]^g \\
           \pi_0(F_{gy}) \ar[r]_{(gy)_\ast}    &  \pi_0(G) \ar[r]_f         &  \pi_0(H) }
\end{displaymath}
of long exact sequences of fibrations in degree $0$. By Lemma \ref{gpdlem} the map $(gy)_\ast$ is monic, so $y_\ast$ is
monic as well. Thus $\pi_0(F_y)$ (resp. $\pi_0(F_{gy})$) is the set-theoretic fibre of the map $\pi_0(f_\ast)$ (resp.
$\pi_0(f)$) for any such choice of $y$. Apply the previous Lemma.
\end{proof}

\section{Short exact sequences associated to torsor trivializations}

The purpose of this section is to illustrate how Grothendieck's short exact sequences
\begin{displaymath}
1 \to \pi_1^{\Gals}(X_Y) \to \pi_1^{\Gals}(X) \to G_Y \to 1
\end{displaymath}
associated to Galois objects $Y \to X$ of $(\et{X}, x)$ arise (at least in part) from analogous sequences in (pointed)
nonabelian $H^1$, and to explain exactly how using pointed nonabelian $H^1_x$ gives different results than unpointed
nonabelian $H^1$ in this context. Grothendieck's proof of these sequences made use of his theory of base change for
fundamental functors in Galois categories $(\mc{C}, F)$ (\S$6$, V, \cite{SGA1}); the methods presented here give a
different interpretation directly in terms of torsors which may be of independent interest.

Suppose $\mc{C}$ is a small Grothendieck site, $H$ any sheaf of groups on $\mc{C}$, and $U \epi \ast$ any representable
sheaf epi of $\mc{C}$. Consider the sequence
\begin{displaymath}
\Htor(\mc{C})_U \hra \Htor(\mc{C}) \xrightarrow{?_{|U}} \Htor(\mc{C}/U)
\end{displaymath}
of groupoids where $\Htor(\mc{C})_U$ is the full subgroupoid of ordinary $H$-torsors on $\mc{C}$ trivializing upon
restriction $?_{|U}$ to $U$. 

\begin{prop} \label{exfull}
For any sheaf of groups $H$ and representable sheaf epi $U \epi \ast$ on a small Grothendieck site $\mc{C}$, there is an
exact sequence
\begin{displaymath}
1 \to \pi_0(\Htor(\mc{C})_U) \hra H^1(\mc{C}, H) \xrightarrow{\pi_0(?_{|U})} H^1(\mc{C}/U, H)
\end{displaymath}
of pointed sets where the indicated map is injective.
\end{prop}
\begin{proof}
An inclusion of a full subgroupoid is always injective on path components so the indicated map is injective. Suppose
$[t] \in \pi_0(\Htor(\mc{C}))$ and $\pi_0(?_{|U})([t]) = [\ast_{\mc{C}/U}]$ so that $[t] \in \ker(\pi_0(?_{|U}))$. Then
$?_{|U}(t) = t_{|U}$ is in the path component of $\ast_{\mc{C}/U}$, so there is an isomorphism $t_{|U} \cong
\ast_{\mc{C}/U}$ which says exactly that $t$ trivializes upon restriction. Therefore $t$ is an object of
$\Htor(\mc{C})_U$, and so is any other object $x \in [t]$ since by definition there is then an isomorphism $x \cong t$
which induces an isomorphism upon restriction. Conversely, if $t$ trivializes upon restriction then by definition there
will be an $H$-equivariant isomorphism $t_{|U} \cong \ast_{|U}$ so that $[t] \in \ker{\pi_0(?_{|U})}$. Therefore
$\ker{\pi_0(?_{|U})} = \pi_0(\Htor(\mc{C})_U)$, so the sequence is exact as was to be shown.
\end{proof}

In particular $(\mc{C}, x)$ may be the pointed \emph{finite} \'etale site $\finet{X}$ for some connected locally
noetherian scheme or DM stack $X$ and $U = Y \epi \ast$ a Galois object of $\mc{C}$. The following justifies a certain
notational identification:

\begin{lem} \label{galident}
For any connected Galois object $Y \epi \ast$ of a Galois category $(\mc{C}, F)$ that is also a connected site there
are bijections
\begin{displaymath}
\pi_0(\Htor(\mc{C})_Y) \cong H^1(G_Y, H)
\end{displaymath}
of pointed sets, natural in constant sheaves of discrete groups $H \deq \Gamma^\ast H$, where the latter pointed set is
the nonabelian $H^1$ of Serre (cf. \cite{Serre1}) for the trivial action of the group $G_Y$ on $H$. There are also
bijections
\begin{displaymath}
H^1(G_Y, H) \cong [BG_Y, BH]
\end{displaymath}
natural in discrete groups $H$ where the righthand side is pointed by the class of the trivial homomorphism.
\end{lem}
\begin{proof}
By Lemma $4$ of \cite{Jardine4} there is a weak equivalence $BH(\ast, BH)_Y \simeq B\Htor(\mc{C})_Y$ where $H(\ast,
BH)_Y$ is the union of all path components of the cocycle category $H(\ast, BH)$ containing cocycles of the form
\begin{displaymath}
\ast \xleftarrow{\simeq} \check{C}(Y) \to BH.
\end{displaymath}
In particular there is a bijection $\pi_0(\Htor(\mc{C})_Y) \cong \pi_0(BH(\ast, BH)_Y)$, and the latter set is bijective
to $\pi(\check{C}(Y), B\Gamma^\ast H)$ as one sees directly from the fact that natural transformations correspond to
simplicial homotopies upon taking nerves (cf. the proof of Lemma $4$ of \cite{Jardine4}). As $H = \Gamma^\ast H$ is a
constant sheaf of groups there are bijections
\begin{displaymath}
\pi(\check{C}(Y), B\Gamma^\ast H) \cong \pi(\ccomp \check{C}(Y), BH) \cong \pi(BG_Y, BH) \cong H^1(G_Y, H),
\end{displaymath}
where $\ccomp \check{C}(Y) \cong BG_Y$ because $Y \times Y \cong \sqcup_{G_Y}{Y}$ since $Y$ is connected Galois. The
latter assertion follows from the bijection $\pi(BG_Y, BH) \cong [BG_Y, BH]$ which exists because $BH$ is fibrant and
all simplicial sets are cofibrant.
\end{proof}

\begin{prop} \label{sestors}
Suppose $(\mc{C}, F)$ is a Galois category that is also a connected Grothendieck site, $Y \epi \ast$ a connected Galois
object of $\mc{C}$ with group $G_Y$, and $H$ a constant sheaf of discrete groups on $\et{X}$. Then there is an exact
sequence
\begin{displaymath}
1 \to H^1(G_Y, H) \hra H^1(X, H) \to H^1(Y, H)
\end{displaymath}
of pointed sets natural in $H$.
\end{prop}
\begin{proof}
Proposition \ref{exfull} applies. Combine the resulting short exact sequence with the identification
$\pi_0(\Htor(\mc{C})_Y) \cong H^1(G_Y, H)$ of the previous Lemma.
\end{proof}

This may be considered to be the unpointed torsor theoretic analogue of Grothendieck's short exact sequence for Galois
descent. Analogous arguments work for pointed torsors: let $F_U$ denote the homotopy fibre of the restriction map
\begin{displaymath}
\Htor_x \xrightarrow{?_{|U}} \Htor_u
\end{displaymath}
for any \emph{pointed} representable sheaf epi $(U, u) \epi \ast$ of a pointed connected small Grothendieck site
$(\mc{C}, x)$ and \emph{constant} sheaf of discrete groups $H \deq \Gamma^\ast H$ as in subsection \ref{fusection}. Then
there is a bijection $\pi_0(F_U) \cong \Hom(\check{C}(U, u), BH)$ by Lemma \ref{cident} and the long exact sequence in
homotopy groups associated to $?_{|U}$ degenerates to the exact sequence
\begin{displaymath}
1 \to \Hom(\check{C}(U, u), BH) \to \pi_0(\Htor_x) \to \pi_0(\Htor_u)
\end{displaymath}
of pointed sets.

\begin{prop} \label{prevprop}
Suppose $(\mc{C}, x)$ is a pointed connected small Grothendieck site, $(U, u) \epi \ast$ a pointed representable
sheaf epi of $(\mc{C}, x)$, and $H \deq \Gamma^\ast H$ a constant sheaf of discrete groups on $\mc{C}$. Then there
are exact sequences
\begin{displaymath}
1 \to \Hom(\check{C}(U, u), BH) \hra H^1_x(\mc{C}, H) \to H^1_u(\mc{C}/(U, u), H)
\end{displaymath}
of pointed sets natural in discrete groups $H$ where the indicated map is injective if $(U, u)$ is connected.
\end{prop}
\begin{proof}
By definition $H^1_x(\mc{C}, H) = \pi_0(\Htor_x)$ and $H^1_u(\mc{C}/(U, u), H) = \pi_0(\Htor_u)$. The indicated map is
injective if $(U, u)$ is connected since then the site $\mc{C}/(U, u)$ is connected so that $\Htor_u$ has no
nontrivial automorphisms at any of its basepoints by Galois theory; then apply Lemma \ref{gpdlem}.
\end{proof}

\begin{cor} \label{corgal}
With the hypotheses of Proposition \ref{prevprop}, if $\mc{C}$ is furthermore a Galois category for the fibre functor
$F_x$ then there are exact sequences
\begin{displaymath}
1 \to \Hom(G_Y, H) \hra H^1_x(\mc{C}, H) \to H^1_y(\mc{C}/(Y, y), H)
\end{displaymath}
natural in discrete groups $H$ associated to any connected Galois object $(Y, y) \epi \ast$ of $\mc{C}$.
\end{cor}
\begin{proof}
By the observation $\ccomp \check{C}(Y, y) \cong BG_Y$.
\end{proof}

As the functors $\Hom(G_Y, -)$, $H^1_x(\mc{C}, -)$, and $H^1_y(\mc{C}/(Y, y), -)$ restricted to $\FinGrp$ are
pro-representable by $G_Y$, $\pi_1^{\Gals}(\mc{C}, x)$, and $\pi_1^{\Gals}(\mc{C}/(Y, y), y)$, respectively, there are
induced exact sequences
\begin{displaymath}
\pi_1^{\Gals}(\mc{C}/(Y, y), y) \to \pi_1^{\Gals}(\mc{C}, x) \epi G_Y \to 1
\end{displaymath}
of profinite groups where the indicated map is an epimorphism. The injectivity of the first map may then be established
by the classical observation that $\pi_1^{\Gals}(\mc{C}/(Y, y), y)$ corresponds to the subgroup of $\Aut(F_x)$ of
automorphisms of $F_x$ fixing $(Y, y)$ ($6.13$, V, \cite{SGA1}).

\section{\'Etale van Kampen theorems}

\'Etale van Kampen theorems have previously been developed with the goal of expressing the pro-groupe fondamental
\'elargi $G$ of any pointed connected small Grothendieck site $(\mc{C}, x)$ in terms of the corresponding pro-groups of
hopefully ``simpler'' spaces. Past approaches include \cite{Stix1}, (\S$5$, IX, \cite{SGA1}) upon which the arguments of
\cite{Stix1} are based, and \cite{Z1}. The former two works are based on rather intricate constructions coming from
descent theory and only treat the case of profinite $\pi_1^{\Gals}$, whereas the latter does address the pro-groupoid
fondamental \'elargi ($4.8$, \cite{Z1}) but only with respect to covers by Zariski open substacks.

The approach described here makes use of homotopy theory to simultaneously hide the descent theoretic aspects of the
former two constructions while generalizing the results of \cite{Z1} to deal with certain non-Zariski coverings. The
resulting van Kampen theorem here is simpler to state than ($5.3$, $5.4$, \cite{Stix1}), does not require the
covering to consist exclusively of monomorphisms, and does not depend on specific properties of the \'etale
topology. On the other hand it is a statement about pro-groups rather than pro-groupoids, so in the latter case one
must still defer to the methods of $\cite{Z1}$ which effectively assume that the covering is given by monomorphisms.

Given any sheaf of groups $H$ on a small Grothendieck site $\mc{C}$, the restrictions $H_{|U}$ of $H$ to the sites
$\mc{C}/U$ for objects $U$ of $\mc{C}$ have associated cocycle categories $H(\ast, BH_{|U})$, and these determine a
simplicial presheaf $B\mathbf{H}(\ast, BH)$ defined in sections $U$ by
\begin{displaymath}
B\mathbf{H}(\ast, BH)(U) \deq BH(\ast, BH_{|U}).
\end{displaymath}
This simplicial presheaf satisfies descent in the sense that it is sectionwise equivalent to a globally fibrant model
(i.e. a fibrant object for the injective closed model structure) and admits a local weak equivalence
\begin{displaymath}
BH \xrightarrow{\simeq} B\mathbf{H}(\ast, BH)
\end{displaymath}
so is a model for the classifying stack of $H$ (Theorem $4$, \cite{Jardine4}).  There is a sectionwise weak equivalence
$B\mathbf{H}(\ast, BH) \simeq B(H_{|?}$-$\Tors)$ where the latter is the simplicial presheaf of nerves of groupoids of
$H_{|U}$-torsors for restrictions of $H$ to the various sites $\mc{C}/U$ (by remarks preceding ibid.), so the latter also
satisfies descent.

Suppose now that $(\mc{C}, x)$ is a pointed small Grothendieck site with a subcanonical topology and finite products and
that $(U, u)$ and $(V, v)$ are pointed objects of $\mc{C}$ such that the pullback square $S$
\begin{displaymath}
\xymatrix{ U \times V \ar[r] \ar[d]  &  U \ar[d] \\
           V \ar[r]  &  \ast }
\end{displaymath}
is a pushout of pointed sheaves on $\mc{C}$, where $U \times V$ is pointed by $u \times v$. By Quilen's axiom SM$7$ (cf.
$11.5$, I, \cite{GJ} for a definition and $3.1$, \cite{Jardine3} for a proof), this pushout determines a homotopy
cartesian square $H_S$
\begin{displaymath}
\xymatrix{ \Htor(\mc{C}) \ar[r] \ar[d] & \Htor(U) \ar[d] \\
           \Htor(V) \ar[r]             & \Htor(U \times V) }
\end{displaymath}
of groupoids of $H$-torsors naturally in $H$ for any \emph{constant} sheaf of groups $H \deq \Gamma^\ast H$ on
$\mc{C}$ whenever $U \to \ast$ or $V \to \ast$ is monic (i.e. a \emph{cofibration} for the injective closed model
structure), or more generally whenever $S$ is homotopy cocartesian for the injective closed model structure.

\begin{lem} \label{hcart}
Suppose a commutative square
\begin{displaymath}
\xymatrix{ X \ar[r] \ar[d]   &  W \ar[d] \\
           Y \ar[r]          &  Z }
\end{displaymath}
is homotopy cartesian over an object $B$ in some right proper closed model category, and suppose the structure maps
from the objects of this square to $B$ are fibrations. Let $c: A \to B$ be a weak equivalence. Then the pullback of
this square along $c$ is also homotopy cartesian.
\end{lem}
\begin{proof}
Factor the map $W \to Z$ as a trivial cofibration $q: W \to Q$ followed by a fibration $g: Q \to Z$, and let $q_c$
(resp. $g_c$) denote the base change of $q$ (resp. $g$) along $c$. Then the composite $Q \to Z \to B$ is a fibration
since $Z \to B$ is a fibration and $W \to B$ is also a fibration so $q_c$ is a weak equivalence by right properness,
and one sees directly that $g_c$ is a fibration. Pulling back $g_cq_c$ along the base change along $c$ of $Y \to Z$
determines a fibration $(Q \times_Z Y) \times_B A \to Y \times_B A$, and it suffices to show that the induced map $X
\times_B A \to (Q \times_Z Y) \times_B A$ is a weak equivalence. This map is the base change along $c$ of the
induced map $X \to Q \times_Z Y$ which is a weak equivalence since the original square was homotopy cartesian. The
composite $Q \times_Z Y \to Y \to B$ is a fibration since $Y \to B$ is a fibration, and $X \to B$ is also a
fibration so the result follows by right properness.
\end{proof}

\begin{prop} \label{ptdhc}
For any constant sheaf of groups $H \deq \Gamma^\ast H$ on a pointed small Grothendieck site $(\mc{C}, x)$ with a
subcanonical topology and finite products and any square $S$ of pointed maps as above such that either one of the sheaf
maps $U \to \ast$ or $V \to \ast$ is monic or $S$ itself is homotopy cocartesian for the injective closed model
structure on $\Shv{\mc{C}}$, there is a homotopy cartesian diagram
\begin{displaymath}
\xymatrix{ \Htor_x(\mc{C}) \ar[r] \ar[d]  &  \Htor_u(U) \ar[d] \\
           \Htor_v(V) \ar[r]              &  \Htor_{u \times v}(U \times V) }
\end{displaymath}
of groupoids of pointed torsors natural in $H$.
\end{prop}
\begin{proof}
There is a diagram
\begin{displaymath}
\xymatrix{ x^\ast(H)\textrm{-}\Tors(\Set) \ar[r]^{?_{|U}} \ar[d]^{?_{|V}}  &
           u^\ast(H)\textrm{-}\Tors(\Set) \ar[d]^{?_{|U \times V}} \\
           v^\ast(H)\textrm{-}\Tors(\Set) \ar[r]^(0.45){?_{|U \times V}}  &  (u \times v)^\ast(H)\textrm{-}\Tors(\Set) }
\end{displaymath}
under $H_S$ induced by $S$ which commutes since the maps of $S$ are \emph{pointed}. In fact, $H = x^\ast(H) =
u^\ast(H_{|U})$ and similarly for the other points, so that all of the groupoids in this diagram are identical to the
groupoid $\Htor(\Set)$, and all of the restriction maps are equalities. One may therefore unambiguously take the
homotopy fibre of $H_S$ over the trivial torsor $H$ of the groupoid $\Htor(\Set)$, and this of course results in the
desired commutative square $H_\ast$ of groupoids of \emph{pointed} $H$-torsors over $x$ by Lemma $1$ of
\cite{Jardine14}.

It remains to be shown that $H_\ast$ is also homotopy cartesian. Let
\begin{displaymath}
h: \ast \to \Htor(\Set)
\end{displaymath}
denote the unique map picking out the trivial $H$-torsor in $\Set$ and factor it as a trivial cofibration $c: \ast \to
C$ followed by a fibration $f: C \to \Htor(\Set)$ (this may be understood in terms of the nerves of the groupoids and
the closed model structure on simplicial sets). A direct calculation shows that pulling back along fibrations sends
homotopy cartesian squares to homotopy cartesian squares, so it suffices to show that pulling back along $c$ preserves
homotopy cartesian squares. Suppose one has a homotopy cartesian square $H_C$ over $C$, and functorially replace all the
maps to $C$ by fibrations; the resulting objects of the new commutative square $H_C^\prime$ are weakly (even homotopy)
equivalent to the old ones and nerves of groupoids are Kan complexes so that $H_C^\prime$ is also homotopy cartesian.
Now apply Lemma \ref{hcart}.
\end{proof}

\begin{lem} \label{proreplem}
Suppose that $D: J \to \Set^{\mc{C}}$ is a finite loop-free diagram of pro-representable functors on a category $\mc{C}$
with all finite limits and colimits. Then
\begin{displaymath}
\lim_{J}{D}
\end{displaymath}
is a functor that is pro-representable by the colimit on $J^{op}$ of the representing pro-objects.
\end{lem}
\begin{proof}
The $J^{op}$-diagram of representing pro-objects is also loop-free, so by the uniform approximation theorem of ($3.3$,
appendix, \cite{AM}) it may be replaced by a cofiltered levelwise diagram of the same shape that is pro-isomorphic to
the original. The finite colimits computed levelwise represent the colimit of the representing pro-objects by ($4.1$,
appendix \cite{AM}). Moreover, at each object $i$ of the new index category $I$ the functor represented by corresponding
levelwise colimit represents the limit over $J$ of the corresponding diagram of representable functors at $i$.
Therefore, the colimits determine a levelwise pro-representation of the limit of the $J$-diagram $D$.
\end{proof}

Finally, the main theorem may be proven:
\begin{thm} \label{mainthm}
Suppose that $(\mc{C}, x)$ is a pointed connected small Grothendieck site with a subcanonical topology, finite limits,
and arbitrary small coproducts, and that $(U, u)$ and $(V, v)$ are pointed connected objects of $\mc{C}$ such that $U
\times V$ is also connected and the pullback square $S$
\begin{displaymath}
\xymatrix{ U \times V \ar[r] \ar[d]  &  U \ar[d] \\
           V \ar[r]  &  \ast }
\end{displaymath}
is a pushout of pointed sheaves on $\mc{C}$, where $U \times V$ is pointed by $u \times v$. Furthermore, suppose that
either one of the sheaf maps $U \to \ast$ or $V \to \ast$ is monic or that $S$ itself is homotopy cocartesian for the
injective closed model structure on $\Shv{\mc{C}}$. Then there is a pushout diagram
\begin{displaymath}
\xymatrix{ \pi_1^{\ets}(\mc{C}/U \times V, u \times v) \ar[r] \ar[d]  &  \pi_1^{\ets}(\mc{C}/U, u) \ar[d] \\
           \pi_1^{\ets}(\mc{C}/V, v) \ar[r]                           &  \pi_1^{\ets}(\mc{C}) }
\end{displaymath}
of ``\'etale'' homotopy pro-groups, each defined in the usual way by means of pointed representable \v{C}ech hypercovers
on its respective site.
\end{thm}
\begin{proof}
The homotopy cartesian squares of Proposition \ref{ptdhc} are clearly natural in $H$ by construction. As $U$, $V$, and
$U \times V$ are connected, the sites $\mc{C}/U$, $\mc{C}/V$, and $\mc{C}/U \times V$ are also connected so that the
respective groupoids of pointed torsors for any constant sheaf of groups $H$ have trivial fundamental groups. Then
Proposition \ref{pbprop} implies that these homotopy cartesian squares together determine a pullback diagram
\begin{displaymath}
\xymatrix{ H^1_x(\mc{C}, -) \ar[r] \ar[d]  &  H^1_u(\mc{C}/U, -) \ar[d] \\
           H^1_v(\mc{C}/V, -) \ar[r]       &  H^1_{u \times v}(\mc{C}/U \times V, -) }
\end{displaymath}
in the functor category $\Set^{\Grp}$. Each of these functors is pro-representable in $\Grp$ by Theorem \ref{charthm}
since all pointed $H$-torsors admit pointed trivializations, so Lemma \ref{proreplem} applies to give the desired
pushout square.
\end{proof}

The word ``\'etale'' is in quotes in the statement of the Theorem because one need not assume that the topology on
$\mc{C}$ is the \'etale topology: the statement still makes sense in any pointed connected site $(\mc{C}, x)$ since the
pro-groups $\pi_1^{\ets}(-)$ are defined by means of pointed connected hypercovers, and the definition of hypercovers is
independent of any particular Grothendieck topology. This Theorem requires no assumptions regarding effective descent
associated to the cover by $U$ and $V$ (cf. \cite{Stix1}) nor is it restricted exclusively to covers by monomorphisms
(as in the case of covers by open substacks; cf. \cite{Z1}).

Of course, in the case where one is actually working in an \'etale site where the ``points'' are defined by geometric
points, one may conclude with a more geometric statement. Here is an example:

\begin{cor} \label{corevc}
Suppose that $X$ is a connected locally noetherian scheme or DM stack, $\et{X}$ the \'etale site of $X$ as defined in
\cite{Z1}, $x: \Omega \to X$ a geometric point of $X$, $(U, u)$ and $(V, v)$ pointed (over $x$) and connected objects of
$\et{X}$ such that $U \times_X V$ is also pointed (by $u \times v$) and connected, and suppose that the pullback square
$S$
\begin{displaymath}
\xymatrix{ U \times V \ar[r] \ar[d]  &  U \ar[d] \\
           V \ar[r]  &  \ast }
\end{displaymath}
is a pushout of sheaves on $\et{X}$. Then, if either one of the sheaf maps $U \to \ast$ or $V \to \ast$ is monic or $S$
itself is homotopy cocartesian for the injective closed model structure on $\Shv{\et{X}}$ then there is a pushout
diagram
\begin{displaymath}
\xymatrix{ \pi_1^{\ets}(U \times V, u \times v) \ar[r] \ar[d]  &  \pi_1^{\ets}(U, u) \ar[d] \\
           \pi_1^{\ets}(V, v) \ar[r]                           &  \pi_1^{\ets}(X, x) }
\end{displaymath}
of \'etale homotopy pro-groups, each defined in the usual way by means of pointed representable \v{C}ech hypercovers
on its respective subsite.
\end{cor}
\begin{proof}
The only condition to check is that pointed $H$-torsors admit representable pointed trivializations for all discrete
groups $H$, but this follows from Lemma \ref{torstriv}.
\end{proof}

The corresponding statement about the usual profinite Grothendieck fundamental groups immediately follows:

\begin{cor} \label{corfin}
With the hypotheses and notation of Corollary \ref{corevc}, there is a pushout square
\begin{displaymath}
\xymatrix{ \pi_1^{\Gals}(U \times V, u \times v) \ar[r] \ar[d]  &  \pi_1^{\Gals}(U, u) \ar[d] \\
           \pi_1^{\Gals}(V, v) \ar[r]  &  \pi_1^{\Gals}(X, x) }
\end{displaymath}
of profinite Grothendieck fundamental groups.
\end{cor}
\begin{proof}
The profinite completion functor $\widehat{\,}\,$ is a left adjoint to the inclusion functor
\begin{displaymath}
i: \textrm{pro-}\FinGrp \hra \textrm{pro-}\Grp
\end{displaymath}
so in particular it preserves pushouts. By Corollary \ref{corevc} one therefore gets a pushout square of the profinite
completions of \'etale fundamental groups, but these are identified with the respective Grothendieck fundamental
groups by Proposition \ref{grothcmp}.
\end{proof}
Importantly, this result is valid even when the pointed covering square $S$ is constructed in $\et{X}$ rather than
$\finet{X}$.

\bibliographystyle{amsalpha}
\bibliography{evc}

\end{document}